\documentclass[11pt]{article}
\usepackage{amsmath}
\usepackage{amssymb}
\usepackage{braket}
\usepackage{url}

\def\C{\mathbb{C}}
\def\RR{\boldsymbol{R}}
\def\BB {\boldsymbol{B}}

\def\qq{\boldsymbol{q}}

\def\rr{\boldsymbol{r}}
\def\sss{\boldsymbol{s}}

\def\BB {\boldsymbol{B}}

\def\AAA {\boldsymbol{A}}

\def\xx{\boldsymbol{x}}
\def\zz{\boldsymbol{z}}
\def\dd{\boldsymbol{d}}
\def\ee{\boldsymbol{e}}
\def\ff{\boldsymbol{f}}
\def\yy{\boldsymbol{y}}
\def\bb{\boldsymbol{b}}
\def\ee{\boldsymbol{e}}
\def\uu{\boldsymbol{u}}

\def\ww{\boldsymbol{w}}
\def\cc{\boldsymbol{c}}

\def\UU{\boldsymbol{U}}

\def\GG{\boldsymbol{G}}

\def\PP{\boldsymbol{P}}

\def\QQ{\boldsymbol{Q}}
\def\RR{\boldsymbol{R}}
\def\TT{\boldsymbol{T}}

\def\MM{\boldsymbol{M}}

\def\aaa{\boldsymbol{a}}

\def\hh{\boldsymbol{h}}
\def\II{\boldsymbol{I}}

\def\00{\boldsymbol{0}}
\def\11{\mathds{1}}
\def\ggamma{\mbox{\boldmath{$\gamma$}}}

\def\rrho{\mbox{\boldmath{$\rho$}}}
\def\xxi{\mbox{\boldmath{$\xi$}}}
\def\eeta{\mbox{\boldmath{$\eta$}}}

\newcommand{\off}[1]{}

\def\MPE{\text{\it{\tiny{MPE}}}}
\def\RRE{\text{\it{\tiny{RRE}}}}
\def\FOM{\text{\it{\tiny{FOM}}}}
\def\GMR{\text{\it{\tiny{GMR}}}}

\newtheorem{theorem}{Theorem}
\newtheorem{lemma}[theorem]{Lemma}

\newenvironment{proof}{\noindent{\bf{Proof.\/}}}{\hfill$\blacksquare$\vskip0.1in}

\newcommand{\beq}{\begin{equation}}
\newcommand{\eeq}{\end{equation}}

\newcommand{\ben}{\begin{enumerate}}
\newcommand{\een}{\end{enumerate}}

\newcommand{\sem}[1]{[ \hspace{-0.02 in} [#1] \hspace{-0.02in} ]}

\begin{document}
\title
{\bf    Minimal Polynomial and Reduced Rank  Extrapolation Methods Are Related}
\author
{Avram Sidi\\
Computer Science Department\\
Technion - Israel Institute of Technology\\ Haifa 32000, Israel\\
e-mail:{~~~}\url{asidi@cs.technion.ac.il}\\
URL:{~~~}\url{http://www.cs.technion.ac.il/~asidi}}
\date{
Appeared in: \ {\em Advances in Computational Mathematics}, 43:151--170, 2017.}
\maketitle
\vspace{-2cm}

\thispagestyle{empty}
\newpage

\begin{abstract}\sloppypar\noindent
Minimal Polynomial Extrapolation (MPE) and Reduced Rank Extrapolation (RRE) are two polynomial methods used for accelerating the  convergence  of  sequences of vectors $\{\xx_m\}$. They are applied successfully in conjunction with fixed-point iterative schemes in the solution of large and sparse systems of linear and  nonlinear equations in different disciplines  of science and engineering.
Both methods produce approximations $\sss_k$ to the limit or antilimit of $\{\xx_m\}$ that are of the form $\sss_k=\sum^k_{i=0}\gamma_i\xx_i$ with $\sum^k_{i=0}\gamma_i=1$, for some scalars $\gamma_i$. The way the two methods are derived suggests that they might,   somehow,  be related to each other; this has not been explored so far, however.
In this work, we tackle this issue and show  that  the vectors $\sss_{k}^\MPE$ and  $\sss_k^\RRE$ produced by the two methods
are related in more than one way,   and independently  of the way the $\xx_m$ are generated.
One of our results states that  RRE stagnates, in the sense that $\sss_k^\RRE=\sss_{k-1}^\RRE$, if and only if  $\sss_{k}^\MPE$ does not exist. Another result states that, when
$\sss_{k}^\MPE$ exists, there holds
$$\mu_k\sss_k^\RRE=\mu_{k-1}\sss_{k-1}^\RRE+
 \nu_k\sss_{k}^\MPE\quad \text{with}\quad \mu_k=\mu_{k-1}+\nu_k,$$ for some positive scalars $\mu_k$, $\mu_{k-1}$,  and $\nu_k$ that depend only on $\sss_k^\RRE$, $\sss_{k-1}^\RRE$, and $\sss_{k}^\MPE$,
 respectively. Our results are valid when MPE and RRE are defined in any weighted inner product and the norm induced by it.
 They  also contain as special cases the known results  pertaining to the  connection between the method of Arnoldi and the method of generalized minimal residuals, two important Krylov subspace methods for solving nonsingular linear systems.
\end{abstract}

\vspace{1cm} \noindent {\bf Mathematics Subject Classification
2000:}  65B05, 65F10, 65F50,  65H10.

\vspace{1cm} \noindent {\bf Keywords and expressions:} Vector extrapolation methods,
minimal polynomial extrapolation (MPE), reduced rank extrapolation (RRE), Krylov subspace methods,
method of  Arnoldi, method of generalized minimal residuals, GMRES.

\thispagestyle{empty}
\newpage
\pagenumbering{arabic}

\section{Introduction}
\label{se1}
{\em Minimal Polynomial Extrapolation} (MPE) of Cabay and Jackson \cite{Cabay:1976:PEM} and {\em Reduced Rank Extrapolation} (RRE) of Kaniel and Stein \cite{Kaniel:1974:LSA}, Eddy  \cite{Eddy:1979:ELV}, and
Me{\u{s}}ina \cite{Mesina:1977:CAI}
are two polynomial methods
of {\em convergence acceleration} or {\em extrapolation} for sequences of vectors.\footnote{The formulations of RRE given in Kaniel and Stein \cite{Kaniel:1974:LSA} and Me{\u{s}}ina
\cite{Mesina:1977:CAI}
 are essentially the same, but they  are  entirely different  from that in Eddy
 \cite{Eddy:1979:ELV}. The  mathematical equivalence of the different formulations is shown in Smith, Ford, and Sidi in \cite{Smith:1987:EMV}.}
They have been used successfully
in different areas of science and engineering
in accelerating the convergence of sequences
that arise, for example,  from application of fixed-point iterative  schemes to large and sparse
linear or nonlinear systems of equations.

These methods and others were  reviewed by Smith, Ford, and Sidi \cite{Smith:1987:EMV}, Sidi, Ford, and Smith \cite{Sidi:1986:ACV} and, more recently, by Sidi \cite{Sidi:2012:RTV}. Their convergence and stability properties were analyzed in the papers by  Sidi \cite{Sidi:1986:CSP}, \cite{Sidi:1994:CIR}, Sidi and Bridger \cite{Sidi:1988:CSA}, and Sidi and Shapira \cite{Sidi:1991:UBC}, \cite{Sidi:1998:UBC}.  Their connection with known Krylov subspace methods for the solution of linear systems of equations was explored in
Sidi \cite{Sidi:1988:EVP}. In Ford and Sidi \cite{Ford:1988:RAV}, they were shown to satisfy certain interesting   recursion relations.  Efficient algorithms for their  implementation that are stable numerically  and economical computationally and storagewise were designed in Sidi \cite{Sidi:1991:EIM}.
Finally, Chapter 4 of the book  by Brezinski and Redivo Zaglia \cite{Brezinski:1991:EMT} is devoted completely to vector extrapolation methods (including MPE and RRE) and their various properties.

 From the way they are derived, one might suspect that MPE and RRE are somehow related. Despite being intriguing and of interest in itself, this  subject   has not been investigated  until now, however. In this work, we undertake precisely this investigation
and show that the two methods are indeed very closely related in more than one way. A partial  description of the results of this investigation are given in the next paragraph.

Let $\{\xx_m\}$ be an {\em arbitrary} sequence of vectors in $\C^N$ endowed with a general {\em weighted} (not necessarily standard Euclidean) inner product and the norm induced by it,  and let $\sss_k^\MPE$ and $\sss_k^\RRE$ be the vectors (approximations to $\lim_{m\to\infty}\xx_m$ when this limit exists, for example) produced by MPE and RRE from the $k+2$ vectors $\xx_0,\xx_1,\ldots,\xx_{k+1}$.
It is known that $\sss_k^\RRE$ always exists, but $\sss_k^\MPE$ may not always exist. One of our results states that,  RRE stagnates, in the sense that \beq\label{eq0}\sss_k^\RRE=\sss_{k-1}^\RRE\quad \Leftrightarrow\quad\sss_{k}^\MPE\ \text{does not exist.}\eeq Another result states that, when
$\sss_{k}^\MPE$ exists, there holds
\beq\label{eq00}\mu_k\sss_k^\RRE=\mu_{k-1}\sss_{k-1}^\RRE+
 \nu_k\sss_{k}^\MPE\quad \text{with}\quad \mu_k=\mu_{k-1}+\nu_k,\eeq for some positive scalars $\mu_k$, $\mu_{k-1}$,  and $\nu_k$ that depend  only on $\sss_k^\RRE$, $\sss_{k-1}^\RRE$, and $\sss_{k}^\MPE$,
 respectively. The precise results and the conditions under which they hold will be given in the next sections.\footnote{Throughout this work, we will use boldface lowercase letters to denote column vectors. In particular, we will denote the zero column vector by $\00$. Similarly,  we will use boldface upper case letters to denote matrices.}

 When the sequence $\{\xx_m\}$ is generated from a {\em  linear} singular system of equations $\xx=\TT\xx+\dd$ via the fixed-point iterative scheme
 $\xx_{m+1}=\TT\xx_m+\dd,$ $m=0,1,\ldots,$ starting with some initial vector $\xx_0$, the vectors $\sss_k^\MPE$ and
  $\sss_k^\RRE$ are precisely those generated by, respectively, the {\em Full Orthogonalization Method} (FOM)   and the method of {\em Generalized Minimal  Residuals} (GMR)---two important Krylov subspace methods for solving linear systems---as these are being applied to the linear system $(\II-\TT)\xx=\dd$, starting with $\xx_0$ as the initial approximation to the solution.
  This is so provided  all four methods are defined using the same weighted inner product and the norm induced by it.\footnote{MPE and RRE were originally defined in $\C^N$ with the standard Euclidean inner product and the norm induced by it.  In subsequent work  by the author and his co-authors, their definitions were generalized by allowing  general inner products and  norms.
The algorithms for implementing MPE and RRE given in \cite{Sidi:1991:EIM} still use the standard Euclidean inner product and the  norm induced by it, however.}

FOM was developed by Arnoldi \cite{Arnoldi:1951:PMI}, who also presented a very elegant algorithm, which employs an interesting process called the {\em Arnoldi--Gram--Schmidt process},  for computing an orthonormal basis for a Krylov subspace. For a discussion of FOM and more, see also Saad \cite{Saad:1981:KSM}.  Different algorithms were given for GMR by Axelsson \cite{Axelsson:1980:CGT}, by Young and Jea \cite{Young:1980:GCG}, by Eisenstat, Elman, and  Schultz \cite{Eisenstat:1983:VIM}, known as {\em Generalized Conjugate Residuals} (GCR),  and  by  Saad and  Schultz \cite{Saad:1986:GMRES}, known as GMRES. GMRES  also uses the Arnoldi--Gram--Schmidt process, and  is known to be the best implementation of GMR.
For Krylov subspace methods in general, see the books by Greenbaum \cite{Greenbaum:1997:IMs}, Saad \cite{Saad:2003:IMS}, and van der Vorst \cite{Vorst:2003:IKM}. The methods FOM and GMRES were also formulated in Essai
\cite{Essai:1998:WFG} in terms of
weighted inner products and norms induced by them; see also G\"{u}ttel and Pestana
\cite{Guttel:2014:SOW}.

Now, there are interesting connections between the vectors generated by FOM and GMR, and by further Krylov subspace methods, and these connections   have been explored in
  Brown \cite{Brown:1991:TCA} and Weiss \cite{Weiss:1990:CBG} originally. This topic has been   analyzed further in the papers by Gutknecht \cite{Gutknecht:1993:CNC}, \cite{Gutknecht:1997:LTS},  Weiss \cite{Weiss:1994:PGC}, Zhou and Walker \cite{Zhou:1994:RST}, Walker \cite{Walker:1995:RSP},
  Cullum and Greenbaum \cite{Cullum:1996:RBG}, and  Eiermann and Ernst \cite{Eiermann:2001:GAT}, by using weighted inner products and norms induced by them.

In view of the mathematical equivalence of MPE to FOM and of RRE to GMR
{\em when $\{\xx_m\}$ is generated from linear systems},
   the results of the present work for MPE and RRE [in particular, \eqref{eq0} and \eqref{eq00}] are precisely   those of \cite{Brown:1991:TCA} and \cite{Weiss:1990:CBG} in the presence of such $\{\xx_m\}$.
      Clearly, our results   pertaining to the relation between MPE and RRE  have a  larger scope than those pertaining to FOM and GMR    because  they apply
    to sequences obtained from nonlinear systems, as well as linear ones, while FOM and GMR apply to linear systems only. Actually, our results  apply to {\em arbitrary}  sequences  $\{\xx_m\}$, independently of how these sequences are generated.  In this sense, the connection between MPE and RRE can be viewed as being of a {\em universal} nature.
We wish to emphasize that (i)\,a priori, it cannot be assumed that MPE and RRE are related when applied to vector sequences $\{\xx_m\}$ arising from nonlinear systems, and (ii)\,in case there is a relationship, it cannot be concluded, a priori,  what form it will assume. In view of this, the fact that MPE and RRE are related as in \eqref{eq0} and \eqref{eq00}  in the presence of {\em arbitrary} sequences $\{\xx_m\}$, whether generated linearly or nonlinearly or otherwise, is quite surprising.

The purpose of this work is twofold:
\begin{enumerate}
\item
In the next section, we (i)\,redefine MPE and RRE using a weighted inner product and the norm induced by it, and (ii)\,develop a {\em unified}  algorithm for their implementation, thus also providing the   theoretical background necessary for  the rest of this work.
We note that these developments are completely new and have not been given before.
They form an  essential part of the proofs of the main results of Section~\ref{se3}. Sometimes, we will refer to the redefined MPE and RRE as {\em weighted} MPE and RRE.
 \item
 In Section \ref{se3}, we state and prove our main results showing that  MPE and RRE, as redefined in Section \ref{se2},  are closely related. Following this, in  Section~\ref{se4}, we discuss the application of our results to sequences $\{\xx_m\}$ generated from a linear nonsingular system of equations via  fixed-point iterative schemes, and show that our main results reduce to the known analogous results of \cite{Brown:1991:TCA} and \cite{Weiss:1990:CBG} that pertain to FOM and GMR, also when all four methods are defined using the same  weighted inner product and the norm induced by it.
 \end{enumerate}

 The weighted inner product $\braket{\cdot\,,\cdot}$ and the  norm $\sem\cdot $ induced by  it (both in $\C^N$) are defined as in
\beq \braket{\yy,\zz}=\yy^*\MM\zz \quad \text{and}\quad  \sem{\zz}=\sqrt{\braket{\zz,\zz}}=\sqrt{\zz^*\MM\zz},\eeq
 where $\MM\in\C^{N\times N}$  is a hermitian positive definite matrix.\footnote{Recall that the  most general inner product in $\C^N$ is the weighted inner product  that is of the form $\braket{\yy,\zz}=\yy^*\MM\zz$, $\MM$ being a hermitian positive definite matrix. Of course, in the simplest case,  $\MM=\text{diag}(\alpha_1,\ldots,\alpha_N)$ with  $\alpha_i>0\ \forall\, i$, so that  $\braket{\yy,\zz}=\sum^N_{i=1}\alpha_i\overline{y_i}\,z_i$ and $\sem{\zz}=(\sum^N_{i=1}\alpha_i|z_i|^2)^{1/2}$. Finally, when $\MM=\II$, we recover the standard Euclidean inner product and the norm induced by it.}
 The matrix $\MM$ is fixed throughout this work.

 For the standard $l_2$ (Euclidean) inner product and the vector norm induced by it, we will use the notation
\beq (\yy,\zz)=\yy^*\zz\quad \text{and}\quad \|\zz\|=\sqrt{\zz^*\zz}.\eeq

A most useful theoretical tool that makes our study of the weighted versions of MPE and RRE run smoothly  is a
generalization of the QR factorization of matrices, which  we call the {\em weighted QR factorization}. This version of the QR factorization seems to have been defined and studied in detail originally in the papers by Gulliksson and Wedin \cite{Gulliksson:1992:MQR} and Gulliksson \cite{Gulliksson:1995:MGS}. It turns out to be the most natural extension of the ordinary QR factorization when orthogonality of two vectors  $\yy,\zz\in\C^N$ is in the sense $\braket{\yy,\zz}=0$.  For convenience, we state the following theorem concerning the weighted QR factorization:
\begin{theorem}\label{th:QRfact} Let
 $$\AAA=[\,\aaa_1\,|\,\aaa_2\,|\,\cdots\,|\,\aaa_s\,]\in\mathbb{C}^{m\times
s},\quad m\geq s,\quad \text{\em rank}(\AAA)=s.$$
Let also $\GG\in\mathbb{C}^{m\times m}$ be  hermitian positive definite
and define the weighted inner product $\braket{\cdot\,,\cdot}$ via   $\braket{\yy,\zz}=\yy^*\GG\zz$.
Then there exist a
 matrix $\QQ\in\mathbb{C}^{m\times s}$, unitary in the sense that $\QQ^*\GG\QQ=\II_s$,
 and an upper
triangular matrix $\RR\in\mathbb{C}^{s\times s}$ with positive
diagonal elements, such that
$$\AAA=\QQ\RR.$$ Specifically,
$$\QQ=[\,\qq_1\,|\,\qq_2\,|\,\cdots\,|\,\qq_s\,],\quad
\RR=\begin{bmatrix}r_{11}&r_{12}&\cdots&r_{1s}\\
&r_{22}&\cdots&r_{2s}\\
&&\ddots&\vdots\\
&&&r_{ss}\end{bmatrix},$$
 $$ \braket{\qq_i,\qq_j}=\qq^*_i\GG\qq_j=\delta_{ij} \quad \forall\ i,j,$$
$$  r_{ij}=\braket{\qq_i,\aaa_j}=\qq^*_i\GG\aaa_j\quad \forall\ i\leq j;\quad r_{ii}>0\quad \forall\ i.$$
 In addition, the matrices $\QQ$ and $\RR$ are unique.
 \end{theorem}

Concerning the computation of $\QQ$ and $\RR$ via the Gram--Schmidt and modified
Gram--Schmidt orthogonalization, see the works mentioned above.

\section{MPE and RRE redefined using a weighted inner \\ {product}} \label{se2}
\setcounter{theorem}{0} \setcounter{equation}{0}
\subsection{General preliminaries}
Let $\{\xx_m\}$ be a vector sequence in $\C^N$. For the sake of argument, we may assume that this sequence results from the fixed-point iterative solution  of the linear or nonlinear system of equations
\beq \label{eq1} \xx=\ff(\xx),\quad \text{solution $\sss$};\quad \xx\in\C^N \quad\text{and}\quad \ff:\C^N\rightarrow \C^N,\eeq that is, from
\beq \label{eq2} \xx_{m+1}=\ff(\xx_m),\quad m=0,1,\ldots,\eeq
 $\xx_0$ being  an initial vector chosen by the user.
Normally, $N$ is  large  and $\ff(\xx)$ is a  sparse vector-valued function.
Now, when the sequence $\{\xx_m\}$ converges, it does so to the solution $\sss$, that is,
$\lim_{m\to\infty}\xx_m=\sss$. In case $\{\xx_m\}$ diverges, we call $\sss$ the antilimit  of $\{\xx_m\}$; vector extrapolation methods in general, and MPE and RRE in particular, may produce sequences of approximations that converge to $\sss$, the antilimit of $\{\xx_m\}$,  in such a case.

Let us define the vectors $\uu_i$  via
\beq \uu_i=\xx_{i+1}-\xx_i, \quad i=0,1,\ldots, \eeq
and the $N\times (k+1)$ matrices $\UU_k$ via
\beq \UU_k=[\,\uu_0\,|\,\uu_1\,|\,\cdots\,|\,\uu_k\,],\quad k=0,1,\ldots\ .\eeq
Of course, there is an integer $k_0\leq N$, such that  the matrices $\UU_k$, $k=0,1,\ldots, k_0-1$,  are of full rank, but   $\UU_{k_0}$ is not; that is,
\beq \label{eq5} \text{rank}\,(\UU_k)=k+1,\quad k=0,1,\ldots,k_0-1;\quad \text{rank}\,(\UU_{k_0})=k_0.\eeq
(Of course, this is the same as saying that $\{\uu_0,\uu_1,\ldots,\uu_{k_0-1}\}$ is a linearly independent set,
but $\{\uu_0,\uu_1,\ldots,\uu_{k_0}\}$ is not.)

Then, both MPE and RRE produce approximations $\sss_k$ (with $k\leq k_0$) to
the solution $\sss$ of \eqref{eq1} that are of the form
\beq \label{eq6}\sss_k=\sum^k_{i=0}\gamma_i\xx_i;\quad \sum^k_{i=0}\gamma_i=1,\eeq
 for some scalars $\gamma_i$. On account of the condition $\sum^k_{i=0}\gamma_i=1$, and because $\xx_i=\xx_0+\sum^{i-1}_{j=0}\uu_j$, we can rewrite \eqref{eq6}  in the form
\beq \label{eq6a}\sss_k=\xx_0+\sum^{k-1}_{j=0}\xi_j\uu_j; \quad \xi_j=\sum^k_{i=j+1}\gamma_i ,\quad
j=0,1,\ldots,k-1,\eeq which can also be expressed in matrix terms as in
\beq \label{eq6b}\sss_k=\xx_0+\UU_{k-1}\xxi,\quad \xxi=[\xi_0,\xi_1,\ldots,\xi_{k-1}]^T.\eeq
We will make use of both representations of $\sss_k$, namely, \eqref{eq6} and \eqref{eq6a}--\eqref{eq6b}, later.
The $\gamma_i$ and $\xi_j$ for MPE are, of course, different from those for RRE, in general.

In the sequel, where confusion may arise, we will denote the vectors $\sss_k$ resulting from MPE and RRE by
$\sss_{k}^\MPE$ and $\sss_{k}^\RRE$, respectively.
Similarly,  to avoid confusion, we will denote the vectors $\ggamma=[\gamma_0,\gamma_1,\ldots,\gamma_k]^T$  and $\xxi=[\xi_0,\xi_1,\ldots,\xi_{k-1}]^T$ corresponding to $\sss_k$ by $\ggamma_k$ (or $\ggamma_k^\MPE$ or $\ggamma_k^\RRE$)  and $\xxi_k$ (or $\xxi_k^\MPE$ or $\xxi_k^\RRE$), respectively, depending on the context.
When necessary, we will also denote (i)\,the $\gamma_i$  associated with $\ggamma_k$  by $\gamma_{ki}$ and (ii)\,the $\xi_j$ associated with $\xxi_k$ by $\xi_{kj}$. That is,
 $$ \ggamma_k=[\gamma_{k0},\gamma_{k1},\ldots,\gamma_{kk}]^T\quad\text{and}\quad
 \xxi_k=[\xi_{k0},\xi_{k1},\ldots,\xi_{k,k-1}]^T.$$

We now describe  how the $\gamma_i$ for MPE and RRE are determined
 when these methods are defined  within the context of $\C^N$ endowed with a weighted inner product and the norm induced by it.

\subsection{Definition of the $\gamma_i$ for MPE and RRE}

 \subsubsection{The $\gamma_i$ for MPE}
  Solve by  least squares the linear overdetermined system of equations
\beq \sum^{k-1}_{i=0}c_i\uu_i=-\uu_k \eeq
for $c_0,c_1,\ldots,c_{k-1}.$  Clearly, this system can be expressed in matrix form as in
\beq \label{eqMPE0}\UU_{k-1}\cc'=-\uu_k;\quad \cc'=[c_0,c_1,\ldots,c_{k-1}]^T,\eeq
and the least squares problem becomes
\beq \label{eqMPE1}\min_{\cc'}\sem {\UU_{k-1}\cc'+\uu_k} .\eeq
Since
$\UU_{k-1}$ has full column rank, this problem has  a unique solution for $\cc'$.  Next,  set $c_k=1$, and  compute
\beq \label{eqMPE2}\gamma_i=\frac{c_i}{\sum^{k}_{i=0}c_i},\quad i=0,1,\ldots,k,\quad
\text{provided}\quad \sum^{k}_{i=0}c_i\neq0.\eeq
   From  this, we see that $\sss_k$ for MPE exists and is unique if and only if $\sum^{k}_{i=0}c_i~\neq~0.$ (Of course, this means that $\sss_k$ for MPE may fail to exist for some $k$ in some cases.)

\subsubsection{The $\gamma_i$ for RRE}
  Solve by  least squares the linear overdetermined system of equations
\beq \sum^k_{i=0}\gamma_i\uu_i=\00, \eeq {subject to the constraint} $\sum^k_{i=0}\gamma_i=1$,
for $\gamma_0,\gamma_1,\ldots,\gamma_k.$ Clearly, this system too  can be expressed in matrix form as in
\beq \label{eqRRE0}\UU_k\ggamma=\00,\quad \ggamma=[\gamma_0,\gamma_1,\ldots,\gamma_k]^T,\eeq
and the constrained least squares problem becomes
\beq \label{eqRRE}\min_{\ggamma}\sem{\UU_k\ggamma}, \quad \text{subject to}\quad \hat{\ee}_k^T\ggamma=1; \quad \hat{\ee}_k=[1,1,\ldots,1]^T\in \C^{k+1}. \eeq
(Here $\hat{\ee}_k$ should not be confused with the $k$th standard basis vector.)
Since $\UU_k$ is of full column rank, this problem has a unique solution for $\ggamma$. From this, it is clear that $\sss_k$ for RRE exists and is unique unconditionally.

\subsection{The special case $k=k_0$}
With $\sss_{k}$ for MPE and RRE already defined, we start with a discussion of the case in which $k=k_0$.

\begin{theorem} Let $\{\xx_m\}$ be an arbitrary sequence, and let MPE and RRE be as defined above.
\begin{enumerate}
\item
 Provided $\sss_{k_0}^\MPE$ exists, we have
$\sss_{k_0}^\MPE=\sss_{k_0}^\RRE$.
\item
Assume the sequence $\{\xx_m\}$ is generated via \eqref{eq1} and \eqref{eq2} with a  linear $\ff(\xx)$, namely, with $\ff(\xx)=\TT\xx+\dd$, where $\TT\in\C^{N\times N}$
 is some constant matrix and $\dd\in\C^N$ is some constant vector, and $(\II-\TT)$ is nonsingular. Then  $\sss_{k_0}^\MPE$ exists, and there holds $\sss_{k_0}^\MPE=\sss_{k_0}^\RRE=\sss$, $\sss$ being the (unique) solution to $\xx=\TT\xx+\dd$. In this case, $k_0$ is the degree of the minimal polynomial of $\TT$ with respect to the vector $\uu_0$.
 \end{enumerate}
\end{theorem}
\begin{proof} We start by observing that the matrix $\UU_{k_0-1}$ has full rank and that the vector $\uu_{k_0}$ is a linear combination of  $\uu_0,\uu_1,\ldots,\uu_{k_0-1}$. As a result, the linear system in \eqref{eqMPE0}
is consistent,  hence
has a unique solution for $\cc'$ in the regular sense, this solution being also the solution to the minimization problem in \eqref{eqMPE1}. Letting $c_{k_0}=1$ and proceeding as in \eqref{eqMPE2}, we obtain the $\ggamma_{k_0}^\MPE$. A similar argument based on \eqref{eqRRE0} and \eqref{eqRRE} shows that $\ggamma_{k_0}^\RRE=\ggamma_{k_0}^\MPE$. This proves part 1 of the theorem.
Part 2 can be proved as in  \cite{Smith:1987:EMV}, for example.\hfill\end{proof}

Since we already know the connection between $\sss_{k_0}^\MPE$ and $\sss_{k_0}^\RRE$, in the sequel, we will consider the cases in which  $k<k_0$ strictly.

\subsection{Determination of the $\gamma_i$ via weighted QR factorization}
A numerically stable and computationally economical algorithm for computing the $\gamma_i$ for both MPE and RRE when $\MM=\II$ has been given in  Sidi \cite{Sidi:1991:EIM}.
A nice feature of this algorithm is that it proceeds via the QR factorization of the matrices $\UU_k$ and unifies the treatments of MPE and RRE. Of course, in order to accommodate the weighted inner product $\braket{\cdot\,,\cdot}$ and the norm $\sem \cdot$ induced by it, we need a different algorithm. Interestingly, an algorithm that is very similar (in fact, identical in form) to the one developed in \cite{Sidi:1991:EIM} can be formulated for this case. This can be accomplished   by  proceeding via the weighted QR factorization of $\UU_k$.
 Even though this algorithm, just as that in \cite{Sidi:1991:EIM},  is designed for {\em computational} purposes, it turns out to be very useful for the {\em theoretical} study of this work concerning the relation between MPE and RRE.  For  some of the details concerning the developments that follow next, we refer the reader to \cite{Sidi:1991:EIM}.

We start with  the weighted QR factorization of $\UU_k$.   Since
$\UU_k$ is of full column rank, by  Theorem \ref{th:QRfact}, it has a unique weighted QR factorization given as in
\beq \label{eq31}\UU_k=\QQ_k\RR_k;\quad
 \QQ_k\in \mathbb{C}^{N\times(k+1)},
 \quad
 \RR_k\in \mathbb{C}^{(k+1)\times(k+1)},
 \eeq
 where $\QQ_k$ is unitary in the sense that  $\QQ_k^*\MM\QQ_k=\II_{k+1}$ since $k<N$,  and $\RR_k$ is upper triangular with positive diagonal elements; that is,
 \beq
  \QQ_k=[\,\qq_0\,|\,\qq_1\,|\,\cdots\,|\,\qq_k\,],
   \quad
\RR_k=\begin{bmatrix}r_{00}&r_{01}&\cdots&r_{0k}\\
&r_{11}&\cdots&r_{1k}\\
&&\ddots&\vdots\\
&&&r_{kk}\end{bmatrix},\eeq
\beq \label{eq33}\qq^*_i\MM\qq_j=\delta_{ij} \quad \forall\ i,j;
 \quad r_{ij}=\qq^*_i\MM\uu_j\quad \forall\ i\leq j;\quad r_{ii}>0\quad \forall\
 i.\eeq
 (Note that, having  positive diagonal elements and being upper triangular, $\RR_k$ is also nonsingular.)
 Clearly, just as $\UU_k$ has the partitioning  $\UU_k=[\UU_{k-1}\,|\,\uu_k\,]$, $\QQ_k$ and $\RR_k$ have the partitionings
 \beq \label{eq17}\QQ_k=[\QQ_{k-1}\,|\,\qq_k] \quad \text{and}\quad
 \RR_k=\left[\begin{array}{c|c}\RR_{k-1} & \rrho_k \\ \hline
 \00^T & r_{kk}\end{array}\right],\quad \rrho_k=[r_{0k},r_{1k},\ldots,r_{k-1,k}]^T.\eeq

We will make use of the following easily verifiable lemma in the sequel:
\begin{lemma} \label{le11} Let
$$ \PP\in\C^{N\times j}\quad \text{and}\quad \PP^*\MM\PP=\II_j.$$
Then
$$ \braket{\PP\yy,\PP\zz}=\yy^*\zz=(\yy,\zz)\quad \text{and}\quad \sem{\PP\zz}=\sqrt{\zz^*\zz}=\|\zz\|.$$
\end{lemma}

By this lemma, for arbitrary $k$, we have
\beq\label{eq991} \braket{\QQ_k\yy,\QQ_k\zz}=\yy^*\zz=(\yy,\zz)\quad \text{and}\quad
\sem{\QQ_k\zz}=\sqrt{\zz^*\zz}=\|\zz\|
\eeq and
\beq \label{eq992} \sem{\UU_k\zz}=\|\RR_k\zz\|.\eeq
Of these, \eqref{eq991} follows from $\QQ_k^*\MM\QQ_k=\II_{k+1}$, while \eqref{eq992} follows from $\UU_k=\QQ_k\RR_k$ and \eqref{eq991}.

 \subsubsection{Determination of $\ggamma_k$ for MPE}
  Let us fix  $c_k=1$ and let $\cc=[c_0,c_1,\ldots,c_k]^T={\left[\begin{array}{c}\cc' \\ \hline 1\end{array}\right]}.$ Then   we have
  $$\UU_{k-1}\cc'+\uu_k=\UU_k\cc \quad \Rightarrow\quad\sem{\UU_{k-1}\cc'+\uu_k}=\sem{\UU_k\cc}=\|\RR_k\cc\|.$$
  As a result,  the minimization problem in \eqref{eqMPE1} becomes,
  $$  \min_{\cc'}\|\RR_k\cc\|. $$
  By \eqref{eq17},
  \beq\label{eq211}\RR_k\cc=
  \left[\begin{array}{c|c}\RR_{k-1} & \rrho_k \\ \hline
 \00^T & r_{kk}\end{array}\right]\left[\begin{array}{c}\cc' \\ \hline 1\end{array}\right]= \left[\begin{array}{c}\RR_{k-1}\cc'+\rrho_k \\ \hline
 r_{kk}\end{array}\right],\eeq which, upon taking norms, yields

 $$ \|\RR_k\cc\|^2=\|\RR_{k-1}\cc'+\rrho_k \|^2+r_{kk}^2.$$
 Clearly, by the fact that $\RR_{k-1}$ is a nonsingular $k\times k$ matrix, the minimum of $\|\RR_k\cc\|$ with respect to $\cc'$ is achieved when $\cc'$ satisfies
 \beq \label{eq18}\RR_{k-1}\cc'+\rrho_k=\00\quad \Rightarrow\quad\RR_{k-1}\cc'=-\rrho_k\quad \Rightarrow \quad \cc'=-\RR_{k-1}^{-1}\rrho_k. \eeq
 Note that $\cc'$ is unique, and so is $\cc$.

 With $\cc'=-\RR_{k-1}^{-1}\rrho_k$, the vector $\ggamma_k$ in MPE is obtained as in
\beq \label{eq20}\ggamma_k^\MPE=\frac{\cc}{\hat{\ee}_k^T\cc},\quad \cc=\left[\begin{array}{c}\cc' \\ \hline 1\end{array}\right]. \eeq
Of course, this is valid only when $\hat{\ee}_k^T\cc=\sum^k_{i=0}c_i\neq 0$, hence only when $\sss_k^\MPE$ exists. The vector $\cc$ exists uniquely whether $\sss_k^\MPE$ exists
or not, however.

\subsubsection{Determination of $\ggamma_k$ for RRE}
 Again by \eqref{eq992}, the minimization problem in \eqref{eqRRE} becomes
    $$\min_{\ggamma}\|\RR_k\ggamma\|, \quad \text{subject to}\quad\hat{\ee}_k^T\ggamma=1, $$ and equivalently,
    $$\min_{\ggamma}\ggamma^*(\RR_k^*\RR_k)\ggamma, \quad \text{subject to}\quad\hat{\ee}_k^T\ggamma=1. $$
    By the lemma in \cite[Appendix A]{Sidi:1991:EIM}, the solution for the vector $\ggamma_k$ in RRE proceeds through the following steps:
\beq \label{eq20a}\RR_k^*\RR_k\hh=\hat{\ee}_k, \quad \hh=[h_0,h_1,\ldots,h_k]^T\quad (\text{solve for $\hh$}).\eeq
\beq \label{eq20b}\lambda =\frac{1}{\sum^k_{i=0}h_i}=\frac{1}{\hat{\ee}_k^T\hh} \quad(\text{$\lambda>0$ always).}\eeq
\beq \label{eq20c}\ggamma_k^\RRE= \lambda \hh. \eeq
Note that  $\hh$ can be determined by solving two  $(k+1)$-dimensional triangular linear systems, namely, (i)\,$\RR_k^*\yy=\hat{\ee}_k$ for $\yy$ and (ii)\,$\RR_k\hh=\yy$ for $\hh$.

For our theoretical study, we need to have $\ggamma_k$ in analytical form. This is  achieved  as follows:
Substituting $\hh=(\RR_k^*\RR_k)^{-1}\hat{\ee}_k$ from \eqref{eq20a} in \eqref{eq20b} and \eqref{eq20c}, we have
\beq \label{eq24}\lambda=\frac{1}{\hat{\ee}_k^T(\RR_k^*\RR_k)^{-1}\hat{\ee}_k}=
\frac{1}{\|\RR_k^{-*}\hat{\ee}_k\|^2}
\eeq and
\beq \label{eq25}\ggamma_k^\RRE=\frac{(\RR_k^*\RR_k)^{-1}\hat{\ee}_k}
{\hat{\ee}_k^T(\RR_k^*\RR_k)^{-1}\hat{\ee}_k}=\frac{\RR_k^{-1}(\RR_k^{-*}\hat{\ee}_k)}
{\|\RR_k^{-*}\hat{\ee}_k\|^2}.\eeq
[Here and in the sequel, $\BB^{-*}$ stands for $(\BB^*)^{-1}=(\BB^{-1}){}^*$.]

\subsection{Unified algorithm for MPE and RRE}
Once the $\ggamma_k$ have been computed as described above, the computation of $\sss_k$ can be achieved via \eqref{eq6a}--\eqref{eq6b} as follows: First, we compute the vector $\xxi_k=[\xi_0,\xi_1,\ldots,\xi_k]^T$  via \eqref{eq6a}, and then, by invoking $\UU_{k-1}=\QQ_{k-1}\RR_{k-1}$, we compute $\sss_k$ via \eqref{eq6b}, as in
\begin{gather}\label{eq6c}
\sss_k=\xx_0+\QQ_{k-1}(\RR_{k-1}\xxi_k)=\xx_0+\sum^{k-1}_{i=0}\eta_i\qq_i; \notag \\  \eeta=\RR_{k-1}\xxi_k,\quad
\eeta=[\eta_0,\eta_1,\ldots,\eta_{k-1}]^T.\end{gather}
For convenience, we give a complete description of the unified algorithm in Table \ref{MPE-RRE}.

\begin{table*}
\caption{\label{MPE-RRE} Unified algorithm for implementing
MPE and RRE.}
\begin{center}
\fbox{\ \quad
\begin{minipage}{5.3in}
\begin{itemize}
\item [\textnormal{Step 0.}] Input: The hermitian positive definite  matrix $\MM\in\C^{N\times N}$, the integer $k$, and the  vectors
$\xx_0,\xx_{1},\ldots,\xx_{k+1}$.
 \item [\textnormal{Step 1.}]
Compute $\uu_i=\Delta \xx_i=\xx_{i+1}-\xx_i$, \  $i=0,1,\ldots,k$.\\
Set $\UU_j=[\uu_0\,|\,\uu_{1}\,|\,\cdots\,|\,\uu_{j}]\in\C^{N\times(j+1)}$, \ $j=0,1,\ldots\ .$\\
Compute the  weighted QR factorization of $\UU_{k}$, namely,
$\UU_{k}=\QQ_{k}\RR_{k}$;\\
$\QQ_k=[\qq_0\,|\,\qq_1\,|\,\cdots\,|\,\qq_k]$ unitary in the sense $\QQ_k^*\MM\QQ_k=\II_{k+1}$, and \\ $\RR_k=[r_{ij}]_{0\leq i,j \leq k}$ upper triangular, \ $r_{ij}=\qq^*_i\MM\uu_j$.\\
($\UU_{k-1}=\QQ_{k-1}\RR_{k-1}$ is contained in $\UU_{k}=\QQ_{k}\RR_{k}$.)
\item [\textnormal{Step 2.}] Computation of  $\ggamma_k=[\gamma_0,\gamma_1,\ldots,\gamma_k]^T$:
\begin{description}
\item For MPE:\\
 Solve the (upper triangular) linear system
$$\RR_{k-1}\cc'=-\rrho_k;\quad
\rrho_k=[r_{0k},r_{1k},\ldots,r_{k-1,k}]^{T}, \quad
\cc'=[c_0,c_1,\ldots,c_{k-1}]^{T}.$$
(Note that $\rrho_k=\QQ^*_{k-1}\MM\uu_{k}$.)\\
Set $c_k=1$ and compute $\alpha=\sum^k_{i=0}c_i.$\\
Set $\ggamma_k=\cc/\alpha$;  that is,   $\gamma_i=c_i/\alpha$, \ $i=0,1,\ldots,k,$ provided $\alpha\neq0$.
 \item For RRE:\\
Solve the linear system
$$\RR_k^*\RR_k\hh=\hat{\ee}_k;\quad \hh=[h_0,h_1,\ldots,h_k]^{T},\quad
\hat{\ee}_k=[1,1,\ldots,1]^{T}\in\C^{k+1}. $$
[This amounts to solving two triangular (lower and upper) systems.]\\
Set $\lambda=\big(\sum^k_{i=0}h_i\big)^{-1}.$ (Note that $\lambda$
is real
and positive.)\\
Set $\ggamma_k=\lambda \hh$; that is, $\gamma_i=\lambda h_i$, \
$i=0,1,\ldots,k.$\end{description}
 \item [\textnormal{Step 3.}] Compute
$\xxi_k=[\xi_0,\xi_1,\ldots,\xi_{k-1}]^{T}$ by
$$\xi_0=1-\gamma_0;\quad \xi_j=\xi_{j-1}-\gamma_j,\quad j=1,\ldots,k-1.$$
Compute $\sss_{k}^{\MPE}$ and $\sss_{k}^{\RRE}$ via
$$\sss_{k}=\xx_0+\QQ_{k-1}\big(\RR_{k-1}\xxi_k\big)=\xx_0+\QQ_{k-1}\eeta.$$
[For this, first compute $\eeta=\RR_{k-1}\xxi_k$,\ \
$\eeta=[\eta_0,\eta_1,\ldots, \eta_{k-1}]^T$. \\
Next, set
$\sss_{k}=\xx_0+\sum^{k-1}_{i=0}\eta_i \qq_i.$]
\end{itemize}
\end{minipage}
}
\end{center}
\end{table*}

\subsection{Error assessment}
Let us now return to the system of equations in \eqref{eq1}. If $\xx$ is an approximation to the solution $\sss$ of this system, then one good measure of the accuracy of $\xx$ is (some norm of) the residual vector $\rr(\xx)$ corresponding to $\xx$ that is given by
\beq \label{resid}\rr(\xx)=\ff(\xx)-\xx.\eeq  This is natural because $\lim_{\xx\to\sss}\rr(\xx)=\rr(\sss)=~\00$.
In case the sequence $\{\xx_m\}$ is being generated
 as in  \eqref{eq2} for solving \eqref{eq1}, our measure for the quality of $\sss_k$ will then be $\rr(\sss_k)$. The following have been shown in \cite{Sidi:1991:EIM}:

\sloppypar
\begin{itemize}
\item
 When $\ff(\xx)$ is linear [that is, $\ff(\xx)=\TT\xx+\dd$ for some constant matrix $\TT\in\C^{N\times N}$ and constant vector $\dd\in\C^N$], $\rr(\sss_k)=\UU_k\ggamma_k$ {\em exactly}.
\item
 When $\ff(\xx)$ is nonlinear, $\UU_k\ggamma_k$ serves as an {\em approximation} to $\rr(\sss_k)$, that is,  $\rr(\sss_k)\approx\UU_k\ggamma_k,$
 and  $\UU_k\ggamma_k$ gets closer and closer to $\rr(\sss_k)$ as convergence is approached.
\end{itemize}

 In addition, for both MPE and RRE, $\sem{\UU_k\ggamma_k}$, the  weighted norm of $\UU_k\ggamma_k$, can be obtained, without actually computing $\UU_k\ggamma_k$ and  taking its norm; it can be obtained very simply  in terms of the quantities already provided by the algorithm we have just described. This is the subject of the next theorem.

 \begin{theorem}\label{th:error}
 The vectors $\UU_k\ggamma_k^\MPE$ and $\UU_k\ggamma_k^\RRE$ satisfy
 \beq \label{resMPERRE}
 \sem{\UU_k\ggamma_k^\MPE}=r_{kk}|\gamma_k|\quad\text{and}\quad
 \sem{\UU_k\ggamma_k^\RRE}=\sqrt{\lambda}.\eeq
  \end{theorem}

  \noindent{\bf Remarks.} \begin{enumerate}
  \item
  Of course, $\gamma_k$ in \eqref{resMPERRE} is $\gamma^\MPE_{kk}$, namely, the last component of the vector $\ggamma_k^\MPE$  corresponding to  $\sss_k^\MPE$. Similarly, $\lambda$ in \eqref{resMPERRE}
  is as defined  in \eqref{eq20b} for $\sss_k^\RRE$.
  \item
   Clearly, \eqref{resMPERRE} is valid for {\em all} sequences $\{\xx_m\}$, whether these are generated by a (linear or nonlinear) fixed-point iterative scheme or otherwise.
  \end{enumerate}

  \begin{proof} By \eqref{eq992}, we have that $\sem{\UU_k\ggamma_k}=\|\RR_k\ggamma_k\|$. Therefore, it is enough to look at $\|{\RR_k\ggamma_k^\MPE}\|$ and $\|{\RR_k\ggamma_k^\RRE}\|$.

  For MPE,  by \eqref{eq211}, \eqref{eq18}, and \eqref{eq20}, with
  $\gamma_k=1/\hat{\ee}_k^T\cc$, we  have
  $$ \RR_k\ggamma_k^\MPE=\frac{1}{\hat{\ee}_k^T\cc}(\RR_k\cc)=
  \gamma_k\left[\begin{array}{c}\00 \\ \hline  r_{kk} \end{array}\right]=
  r_{kk} \gamma_k\left[\begin{array}{c}\00 \\ \hline  1 \end{array}\right].$$
  Taking norms on both sides,  we obtain the result for MPE.

   As for RRE, by \eqref{eq25}, we have
  $$\RR_k\ggamma_k^\RRE=
\frac{\RR_k^{-*}\hat{\ee}_k}
{\|\RR_k^{-*}\hat{\ee}_k\|^2}.$$
Taking norms on both sides, and invoking \eqref{eq24}, we obtain the result for RRE.   \hfill\end{proof}

 \section{MPE and RRE are related} \label{se3}
 \setcounter{theorem}{0} \setcounter{equation}{0}
 We now turn to the study of the relation between MPE and RRE. We do this by analyzing the vectors $\UU_k\ggamma_k$ for both methods.
We begin by restating that since
\beq \label{eq3-2}\UU_k\ggamma_k=\QQ_k(\RR_k\ggamma_k)\quad\text{and}\quad
\sem{\UU_k\ggamma_k}=\|\RR_k\ggamma_k\|,\eeq and since $\QQ_k$ and $\RR_k$ are  the same  for both MPE and RRE, the vector that is of relevance for both methods is $\RR_k\ggamma_k$, and we turn to the study of this vector. In addition, we express everything in terms of the vectors $\cc'$ and $\cc$ and the matrices $\QQ_k$ and  $\RR_k$, which do not depend either on $\sss_k^\MPE$ or $\sss_k^\RRE$.  In the developments that follow, we will also recall that
$\|\yy\|=\sqrt{\yy^*\yy}$ always.

\subsection{$\RR_k\ggamma_k$ for MPE and RRE and an identity} \label{sse31}
Assuming that $\sss_k^\MPE$ exists, hence $\hat{\ee}_k^T\cc\neq0$, by
\eqref{eq20}, we first have
$$ \RR_k\ggamma_k^\MPE=\frac{1}{\hat{\ee}_k^T\cc}\RR_k\cc,$$
which, upon invoking \eqref{eq211} and \eqref{eq18}, becomes
\beq\label{eq3-3}
\RR_k\ggamma_k^\MPE =\frac{r_{kk}}{\hat{\ee}_k^T\cc }
  \left[\begin{array}{c}\00\\ \hline 1\end{array}\right].\eeq
  Of course, this immediately implies that
\beq \label{eq3-4}\|\RR_k\ggamma_k^\MPE\|=\frac{r_{kk}}{|\hat{\ee}_k^T\cc|}.\eeq

As for RRE, by \eqref{eq25}, we have
\beq \label{eq3-5}\RR_k\ggamma_k^\RRE=
\frac{\RR_k^{-*}\hat{\ee}_k}
{\|\RR_k^{-*}\hat{\ee}_k\|^2}.\eeq
Of course, this immediately implies that

\beq \label{eq3-6}\|\RR_k\ggamma_k^\RRE\|=\frac{1}{\|\RR_k^{-*}\hat{\ee}_k\|}\quad \Rightarrow\quad
\RR_k^{-*}\hat{\ee}_k=\frac{\RR_k\ggamma_k^\RRE}{\|\RR_k\ggamma_k^\RRE\|^2}.\eeq

We now go on to study  $\RR_k^{-*}\hat{\ee}_k$ in more detail. First, by \eqref{eq17} and \eqref{eq18},
\beq \RR_k^{-1}=\left[\begin{array}{c|c}\RR_{k-1}^{-1} & \cc'/r_{kk} \\ \hline
 \00^T & 1/r_{kk}\end{array}\right]\quad \Rightarrow\quad
 \RR_k^{-*}=\left[\begin{array}{c|c}\RR_{k-1}^{-*} & \00\\ \hline
 \cc'{}^*/r_{kk}& 1/r_{kk}\end{array}\right].\eeq
  Consequently, invoking also  $\hat{\ee}_k=\left[\begin{array}{c}\hat{\ee}_{k-1}\\ \hline 1\end{array}\right]$, we have
\begin{align} \RR_k^{-*}\hat{\ee}_k&=\left[\begin{array}{c|c}\RR_{k-1}^{-*} & \00\\ \hline
 \cc'{}^*/r_{kk}& 1/r_{kk}\end{array}\right]
 \left[\begin{array} {c}\hat{\ee}_{k-1} \\
 \hline 1\end{array}\right] \notag \\
 &=
 \left[\begin{array}{c} \RR_{k-1}^{-*}\hat{\ee}_{k-1}\\
\hline \cc'{}^*\hat{\ee}_{k-1}/r_{kk}+1/r_{kk}\end{array}\right] \notag \\
&=
 \left[\begin{array}{c} \RR_{k-1}^{-*}\hat{\ee}_{k-1}\\
\hline  \overline{\hat{\ee}_{k}^T\cc}/r_{kk}\end{array}\right], \label{eq3-8a}
 \end{align}
 which, by \eqref{eq3-6},  can also be expressed as in
\beq \label{eq3-8}\frac{1} {\|\RR_k\ggamma_k^\RRE\|^2}\RR_k\ggamma_k^\RRE=\frac{1}
{\|\RR_{k-1}\ggamma_{k-1}^\RRE\|^2}
  \left[\begin{array}{c} \RR_{k-1}\ggamma_{k-1}^\RRE\\ \hline 0\end{array}\right]+
\frac{\overline{\hat{\ee}_{k}^T\cc}}{r_{kk}}\left[\begin{array}{c} \00 \\
\hline 1 \end{array}\right]. \eeq Clearly, \eqref{eq3-8} is  an {\em identity} for  RRE relating  $\sss_{k-1}^\RRE$ and $\sss_{k}^\RRE$;
 we will make use of it  in the developments of the next subsection. (Here $\overline{t}$ stands for the complex conjugate of~$t$.)

\noindent{\bf Remark.} Recall that the vector $\cc$ exists uniquely for all $k<k_0$. Thus, \eqref{eq3-8} is valid whether $\sss_k^\MPE$ exists or not.

 \subsection{Main results}

 The following theorem is our first main result, and  concerns the case in which $\sss_k^\MPE$ does not exist and RRE stagnates.
 \begin{theorem} \label{th1}
 \begin{enumerate} \item
 In case $\sss_k^\MPE$ does not exist,  there holds
 \beq \label{eq3-10}\sss_k^\RRE=\sss_{k-1}^\RRE,\eeq

 which also implies
  \beq \label{eq3-10c}\UU_k\ggamma_k^\RRE=\UU_{k-1}\ggamma_{k-1}^\RRE. \eeq
 \item
 Conversely, if \eqref{eq3-10} holds, then $\sss_k^\MPE$ does not exist.
 \end{enumerate}
 \end{theorem}
 \begin{proof} The proof is based on the fact that $\sss_k^\MPE$ exists if and only if
 $\hat{\ee}_{k}^T{\cc}\neq0$.

 \noindent{\em Proof of part 1:} Since $\hat{\ee}_{k}^T{\cc}=0$ when   $\sss_k^\MPE$ does not exist,
 by \eqref{eq3-8},
 \beq \label{eq3-11}\frac{1} {\|\RR_k\ggamma_k^\RRE\|^2}\RR_k\ggamma_k^\RRE=\frac{1}
{\|\RR_{k-1}\ggamma_{k-1}^\RRE\|^2}
  \left[\begin{array}{c} \RR_{k-1}\ggamma_{k-1}^\RRE\\ \hline 0\end{array}\right].\eeq
  Taking Euclidean norms in \eqref{eq3-11}, we obtain
\beq \label{eq3-12}\|\RR_k\ggamma_k^\RRE\|=\|\RR_{k-1}\ggamma_{k-1}^\RRE\|,\eeq
which, upon substituting back in \eqref{eq3-11}, gives
\beq \label{eq3-9}
\RR_k\ggamma_k^\RRE=\left[\begin{array}{c} \RR_{k-1}\ggamma_{k-1}^\RRE \\
\hline 0\end{array}\right]=\RR_k\left[\begin{array}{c} \ggamma_{k-1}^\RRE \\
\hline 0\end{array}\right].  \eeq
By  the fact that  $\RR_k$ is nonsingular, it follows that
\beq \label{eq3-15}\ggamma_k^\RRE=\left[\begin{array}{c}\ggamma_{k-1}^\RRE \\
\hline 0\end{array}\right],\eeq which, together with \eqref{eq6}, gives \eqref{eq3-10}.
\medskip

\noindent{\em Proof of part 2:} By \eqref{eq3-10} and \eqref{eq6b},
we have
\beq \sss_k^\RRE=\xx_0+\UU_{k-1}\xxi_k^\RRE=\xx_0+\UU_{k-2}\xxi_{k-1}^\RRE=
 \sss_{k-1}^\RRE,\eeq from which
 \beq  \label{eq3-18a} \UU_{k-1}\xxi_k^\RRE=\UU_{k-2}\xxi_{k-1}^\RRE\quad \Rightarrow\quad
 \UU_{k-1}\xxi_k^\RRE=\UU_{k-1} \left[\begin{array}{c}\xxi_{k-1}^\RRE\\ \hline 0\end{array}\right].\eeq
By the fact that  $\UU_{k-1}$ is of full column rank, \eqref{eq3-18a} implies that
\beq \xxi_k^\RRE=\left[\begin{array}{c}\xxi_{k-1}^\RRE\\ \hline 0\end{array}\right],\eeq
which, when combined with the relation [$\xi_{kj}=\sum^k_{i=j+1}\gamma_{ki}$, by which,
$\xi_{k,k-1}=\gamma_{kk}$]  in \eqref{eq6a}, gives
\eqref{eq3-15}.
Multiplying both sides of \eqref{eq3-15} on the left by $\RR_k$, we obtain
\beq  \label{eq3-20a}\RR_k\ggamma_k^\RRE=\left[\begin{array}{c}\RR_{k-1}\ggamma_{k-1}^\RRE\\ \hline 0\end{array}\right]\quad \Rightarrow\quad \|\RR_k\ggamma_k^\RRE\|=\|\RR_{k-1}\ggamma_{k-1}^\RRE\|.\eeq
Substituting \eqref{eq3-20a} in \eqref{eq3-8}, we obtain $\hat{\ee}_k^T\cc=0$, and this completes the proof. \end{proof}

\noindent{\bf Remark.} What Theorem \ref{th1} is saying is that the stagnation of RRE (in the sense that $\sss_k^\RRE=\sss_{k-1}^\RRE$) and the failure of $\sss_k^\MPE$ to exist take place simultaneously. In addition, this phenomenon  is of a universal nature because it is independent of how the sequence $\{\xx_m\}$ is generated.\\

 The next theorem is our second main result, and concerns the general case in which $\sss_k^\MPE$ exists.

 \begin{theorem} \label{th2}
 In case $\sss_k^\MPE$ exists,  there hold
 \beq \label{eq3-16}
 \frac{1}{\sem{\UU_k\ggamma_k^\RRE}^2}=\frac{1}{\sem{\UU_{k-1}\ggamma_{k-1}^\RRE}^2}+\frac{1}{\sem{\UU_k\ggamma_k^\MPE}^2}\eeq
and
 \beq \label{eq3-17}
 \frac{\UU_k\ggamma_k^\RRE}{\sem{\UU_k\ggamma_k^\RRE}^2}=
 \frac{\UU_{k-1}\ggamma_{k-1}^\RRE}{\sem{\UU_{k-1}\ggamma_{k-1}^\RRE}^2}
 +\frac{\UU_k\ggamma_k^\MPE}{\sem{\UU_k\ggamma_k^\MPE}^2}.
 \eeq
 Consequently, we also have
 \beq \label{eq3-18}
 \frac{\sss_k^\RRE}{\sem{\UU_k\ggamma_k^\RRE}^2}=
 \frac{\sss_{k-1}^\RRE}{\sem{\UU_{k-1}\ggamma_{k-1}^\RRE}^2}
 +\frac{\sss_k^\MPE}{\sem{\UU_k\ggamma_k^\MPE}^2}.
 \eeq
 In addition,
 \beq\label{eq3-55} \sem{\UU_k\ggamma_k^\RRE}<\sem{\UU_{k-1}\ggamma_{k-1}^\RRE}.\eeq

 \end{theorem}
 \begin{proof} Since $\sss_k^\MPE$ exists, we have $\hat{\ee}_{k}^T{\cc}\neq0$.
 Taking the Euclidean  norm of both sides in \eqref{eq3-8}, and observing that the two terms on the right-hand side are orthogonal to each other in the Euclidean  inner product, we first obtain
 \beq \label{eq-19}
 \frac{1}{\|\RR_k\ggamma_k^\RRE\|^2}=\frac{1}{\|\RR_{k-1}\ggamma_{k-1}^\RRE\|^2}+
 \bigg(\frac{|\hat{\ee}_k^T\cc|}{r_{kk}}\bigg)^2,\eeq
  which, upon invoking \eqref{eq3-4}, gives
 \beq \label{eq3-20}
 \frac{1}{\|\RR_k\ggamma_k^\RRE\|^2}=\frac{1}{\|\RR_{k-1}\ggamma_{k-1}^\RRE\|^2}+
 \frac{1}{\|\RR_k\ggamma_k^\MPE\|^2}.\eeq
 The result in   \eqref{eq3-16} follows from \eqref{eq3-20} and \eqref{eq3-2}.

 Next, invoking \eqref{eq3-3} and \eqref{eq3-4} in  \eqref{eq3-8}, we obtain
\beq \label{eq3-22}\frac{1}{\|\RR_k\ggamma_k^\RRE\|^2}\RR_k\ggamma_k^\RRE=
\frac{1}{\|\RR_{k-1}\ggamma_{k-1}^\RRE\|^2}\left[\begin{array}{c} \RR_{k-1}\ggamma_{k-1}^\RRE\\
\hline 0\end{array}\right]
+\frac{1}{\|\RR_k\ggamma_k^\MPE\|^2}\RR_k\ggamma_k^\MPE.\eeq
Multiplying both sides of \eqref{eq3-22} on the left by $\QQ_k$,  and invoking \eqref{eq3-2} and
\beq \QQ_k\left[\begin{array}{c} \RR_{k-1}\ggamma_{k-1}^\RRE \\
\hline 0\end{array}\right]=[\,\QQ_{k-1}\,|\,\qq_k\,]\left[\begin{array}{c} \RR_{k-1}\ggamma_{k-1}^\RRE \\
\hline 0\end{array}\right] =\QQ_{k-1}(\RR_{k-1}\ggamma_{k-1}^\RRE)=\UU_{k-1}\ggamma_{k-1}^\RRE,
\eeq
 we  obtain \eqref{eq3-17}.

 Let us rewrite \eqref{eq3-17} in the form
 \beq \label{eq3-23}
\frac{1}{\sem{\UU_k\ggamma_k^\RRE}^2}\UU_k\ggamma_k^\RRE=
 \frac{1}{\sem{\UU_{k-1}\ggamma_{k-1}^\RRE}^2}\UU_{k}
 \left[\begin{array}{c}\ggamma_{k-1}^\RRE\\ \hline 0 \end{array}\right]
 +\frac{1}{\sem{\UU_k\ggamma_k^\MPE}^2}\UU_k\ggamma_k^\MPE.
 \eeq
From \eqref{eq3-23} and by the fact that  $\UU_k$ is of full column rank, it follows that
\beq
\frac{1}{\sem{\UU_k\ggamma_k^\RRE}^2}\ggamma_k^\RRE=
 \frac{1}{\sem{\UU_{k-1}\ggamma_{k-1}^\RRE}^2}
 \left[\begin{array}{c}\ggamma_{k-1}^\RRE\\ \hline 0 \end{array}\right]
 +\frac{1}{\sem{\UU_k\ggamma_k^\MPE}^2}\ggamma_k^\MPE,\eeq
 and this, together with \eqref{eq6}, gives \eqref{eq3-18}.

 Finally, \eqref{eq3-55} follows directly from \eqref{eq3-16}.
  \end{proof}

 The following facts can be deduced directly from \eqref{eq3-16}:
\beq \label{eq91}\sem{\UU_k\ggamma_k^\MPE}=\frac{\sem{\UU_k\ggamma_k^\RRE}}
{\sqrt{1-(\sem{\UU_k\ggamma_k^\RRE}/\sem{\UU_{k-1}\ggamma_{k-1}^\RRE})^2}}\quad
\text{when $\sss_k^\MPE$ exists}.\eeq

\beq \label{eq92}
 \frac{1}{\sem{\UU_k\ggamma_k^\RRE}^2}= \sum_{i\in S_k}\frac{1}{\sem{\UU_i\ggamma_i^\MPE}^2};\quad S_k=\{0\leq i\leq k:\
\sss_i^\MPE\ \text{exists}\}.\eeq

\subsection{Implications of Theorems \ref{th1} and \ref{th2}}
Let us go back to the case in which $\{\xx_m\}$ is generated as in $\xx_{m+1}=\ff(\xx_m)$, $m=0,1,\ldots,$ from the  system $\xx =\ff(\xx)$. As we have already noted, with the residual associated with  an arbitrary vector $\xx$ defined as $\rr(\xx)=\ff(\xx)-\xx$, (i)\,$\UU_k\ggamma_k=\rr(\sss_k)$ when $\ff(\xx)$ is linear, and (ii)\,$\UU_k\ggamma_k\approx\rr(\sss_k)$  when $\ff(x)$ is nonlinear and $\sss_k$ is close to the solution $\sss$ of $\xx=\ff(\xx)$.  Then, Theorem \ref{th2} [especially \eqref{eq91}] implies that the convergence behaviors of MPE and RRE are interrelated in the following sense:  MPE and RRE either converge  well simultaneously or perform poorly simultaneously.
Letting $\phi_k^\MPE=\sem{\UU_k\ggamma_k^\MPE}$ and $\phi_k^\RRE=\sem{\UU_k\ggamma_k^\RRE}$, and recalling  that $\phi_k^\RRE/\phi_{k-1}^\RRE~\leq~1$ for all $k$, we have the following:
(i)\, When $\phi_k^\RRE/\phi_{k-1}^\RRE$ is significantly smaller than~$1$, which means that RRE is performing well,
$\phi_k^\MPE$ is close to $\phi_k^\RRE$, that is, MPE is performing well too,  and (ii)\,when $\phi_k^\MPE$ is  increasing, that is, MPE is performing poorly, $\phi_k^\RRE/\phi_{k-1}^\RRE$ is approaching $1$, that is, RRE is performing poorly too. Thus, when the graph of $\phi_k^\MPE$ has a peak for $\tilde{k}_1\leq k\leq \tilde{k}_2$, then the graph of $\phi_k^\RRE$  has a plateau for $\tilde{k}_1\leq k\leq \tilde{k}_2$.
This is known as the {\em peak-plateau} phenomenon in the context of Krylov subspace methods for linear systems.

\section{Connection with Krylov subspace methods and concluding remarks} \label{se4}
\setcounter{theorem}{0} \setcounter{equation}{0}
\subsection{MPE and RRE on linear systems}
Consider again the linear system of equations $\xx=\TT\xx+\dd$, where the matrix $(\II-\TT)$
is nonsingular, and generate $\{\xx_m\}$ via $\xx_{m+1}=\TT\xx_m+\dd$, $m=0,1,\ldots,$ with some initial vector $\xx_0$. Apply MPE and RRE to $\{\xx_m\}$ to obtain the vectors $\sss_k$ as before. As
already stated, $\UU_k\ggamma_k=\rr_k=\rr(\sss_k)$, where $\rr(\xx)=
(\TT\xx+\dd)-\xx$ is the residual vector for the system $(\II-\TT)\xx=\dd$ associated with $\xx$.
In this case, we have the next theorem as  a corollary of  Theorems \ref{th1} and \ref{th2}:

\begin{theorem}\label{th3} Let  the  sequence $\{\xx_m\}$ be generated recursively via $\xx_{m+1}=\TT\xx_m+\dd$, $m=0,1,\ldots,$ the matrix $(\II-\TT)$ being nonsingular.
Let also $\rr(\xx)=\TT\xx+\dd-\xx$ be the residual vector corresponding to $\xx$. Let $k_0$ be the degree of the minimal polynomial of $\TT$ with respect to $\uu_0=\xx_1-\xx_0$. Then, for $k<k_0$,
the vectors $\sss_k^\MPE$ and $\sss_k^\RRE$ obtained by applying MPE and RRE to $\{\xx_m\}$ and their residual vectors $\rr(\sss_k^\MPE)=\rr_k^\MPE$ and
$\rr(\sss_k^\RRE)=\rr_k^\RRE$ satisfy the following for this special case:
\begin{enumerate}
\item $\sss_k^\RRE=\sss_{k-1}^\RRE$ if and only if $\sss_k^\MPE$ fails to exist.
\item In case $\sss_k^\MPE$ exists, there hold
\beq \label{eq222}
\frac{1}{\sem{\rr_k^\RRE}^2}=
 \frac{1}{\sem{\rr_{k-1}^\RRE}^2}
 +\frac{1}{\sem{\rr_k^\MPE}^2}.
 \eeq

\beq \label{eq223}
\frac{\rr_k^\RRE}{\sem{\rr_k^\RRE}^2}=
 \frac{\rr_{k-1}^\RRE}{\sem{\rr_{k-1}^\RRE}^2}
 +\frac{\rr_k^\MPE}{\sem{\rr_k^\MPE}^2}.
 \eeq
 Consequently, we also have
 \beq \label{eq224}
 \frac{\sss_k^\RRE}{\sem{\rr_k^\RRE}^2}=
 \frac{\sss_{k-1}^\RRE}{\sem{\rr_{k-1}^\RRE}^2}
 +\frac{\sss_k^\MPE}{\sem{\rr_k^\MPE}^2}.
 \eeq
 In addition,
 \beq\label{eq225} \sem{\rr_k^\RRE}<\sem{\rr_{k-1}^\RRE}.\eeq
\item $\sss_{k_0}^\MPE=\sss_{k_0}^\RRE=\sss$, where $\sss$ is the solution to $(\II-\TT)\xx=\dd$.
\end{enumerate}
\end{theorem}

In view of \eqref{eq222}, the results in \eqref{eq91} and \eqref{eq92} become
\beq \label{eq911}\sem{\rr_k^\MPE}=\frac{\sem{\rr_k^\RRE}}
{\sqrt{1-(\sem{\rr_k^\RRE}/\sem{\rr_{k-1}^\RRE})^2}}
\quad
\text{when $\sss_k^\MPE$ exists}\eeq and
\beq \label{eq922}
 \frac{1}{\sem{\rr_k^\RRE}^2}= \sum_{i\in S_k}\frac{1}{\sem{\rr_i^\MPE}^2};\quad S_k=\{0\leq i\leq k:\
\sss_i^\MPE\ \text{exists}\}.\eeq

\subsection{Equivalence of redefined  MPE and RRE to Krylov subspace methods for linear systems}
 Theorem 2.4 in  \cite{Sidi:1988:EVP} concerns  the mathematical equivalence of
 vector extrapolation methods to Krylov subspace methods, when all these methods are defined
 using the standard Euclidean  inner product $(\cdot\,,\cdot)$ and the standard  norm $\|\cdot\|$ induced by $(\cdot\,,\cdot)$:  This theorem states specifically that
MPE  and RRE are equivalent  to, respectively,  the  full orthogonalization method (FOM) of Arnoldi and  the method of generalized minimal residuals (GMR) when
\begin{itemize}\item
MPE and RRE  are being applied to the sequence $\{\xx_m\}$ obtained via $\xx_{m+1}=\TT\xx_m+\dd$, $m=0,1,\ldots,$ with some $\xx_0$,   and
 \item
 FOM and GMR are being applied to $(\II-\TT)\xx=\dd$, starting with the same initial vector $\xx_0$.
     \end{itemize}
    As  stated in Theorem \ref{th55} below, this theorem holds true also when MPE, RRE, FOM, and GMR are defined using the weighted inner product $\braket{\cdot\,,\cdot}$ and the weighted norm  $\sem{\cdot}$ induced by
    $\braket{\cdot\,,\cdot}$. In the next paragraph, we state these definitions of
FOM and GMR.

For  a nonsingular linear system $\AAA\xx=\bb$, whose solution we denote by $\sss$, FOM  and GMR construct  their approximations $\ww_k$ to $\sss$ as follows:
Define the residual vector  corresponding to $\xx$ by $\rr(\xx)=\bb-\AAA\xx$ and denote
$\rr_0=\rr(\xx_0)$ for some initial vector $\xx_0$. Let ${\cal K}_k(\AAA;\rr_0)=\text{span}\{\rr_0,\AAA\rr_0,\ldots, \AAA^{k-1}\rr_0\}$. Then, for each method, the approximation $\ww_k$ to $\sss$ is of the form $\ww_k=\xx_0+\yy$ such that $\yy\in{\cal K}_k(\AAA;\rr_0)$,
and $\yy$ is the vector to be determined. Using the weighted inner product $\braket{\cdot\,,\cdot}$ and  the norm $\sem{\cdot}$ induced by it, these methods can be redefined as follows: \begin{itemize}
 \item
 For FOM, $\yy$ is determined by requiring that $\braket{\zz,\rr_k^\FOM}=0$ for all $\zz\in {\cal K}_k(\AAA;\rr_0)$, where $\rr_k^\FOM=\rr(\ww_k^\FOM)$.
 \item
For GMR, $\yy$ is determined by requiring that $\sem{\rr_k^\GMR}=
\min_{\yy\in{\cal K}_k(\AAA;\rr_0)}\sem{\rr(\xx_0+\yy)}$, where $\rr_k^\GMR=\rr(\ww_k^\GMR)$.
\end{itemize}
Then we have the following generalization of Theorem 2.4 in \cite{Sidi:1988:EVP}:

\begin{theorem} \label{th55}
Consider the nonsingular linear system $(\II-\TT)\xx=\dd$. Apply FOM and GMR to this system starting with some initial vector $\xx_0$. Apply MPE and RRE to the sequence $\{\xx_m\}$ obtained from $\xx_{m+1}=\TT\xx_m+\dd$, $m=0,1,\ldots,$ with the same initial vector $\xx_0$. Then
\beq \ww_k^\FOM=\sss_k^\MPE\quad\text{and}\quad \ww_k^\GMR=\sss_k^\RRE,\eeq
when all four methods are defined using the same weighted inner product $\braket{\cdot\,,\cdot}$ and the norm $\sem{\cdot}$ induced by it. Consequently, all of the results of Theorem \ref{th3} apply  verbatim to the vectors
$\ww_k^\FOM$ and $\ww_k^\GMR$.
\end{theorem}
\begin{proof} The same as that of \cite[Theorem 2.4]{Sidi:1988:EVP}.
\end{proof}

In view of Theorem \ref{th55}, Theorem \ref{th3} holds verbatim with  $\sss_k^\MPE$, $\rr_k^\MPE$ and $\sss_k^\RRE$, $\rr_k^\RRE$ there replaced by $\ww_k^\FOM$, $\rr_k^\FOM$  and
$\ww_k^\GMR$, $\rr_k^\GMR$, respectively. Of course, these  results for FOM and GMR are not new. As already mentioned,  they  were given originally by Weiss \cite{Weiss:1990:CBG} and by Brown \cite{Brown:1991:TCA}, and developed further in the papers mentioned in Section~\ref{se1}.

Note that the vectors $\ww_k^\FOM$  and $\ww_k^\GMR$ can be obtained numerically by modifying the known algorithms for FOM and GMR such that the Euclidean inner product and the associated  norm are replaced by a weighted
inner product and the associated  norm. This is precisely what is done  in the paper by Essai \cite{Essai:1998:WFG}, which was mentioned in Section \ref{se1}.


\end{document}